\documentclass[11pt]{amsart}

\newtheorem{fed}{\textbf{Definition}}[section]
\newtheorem{thm}[fed]{\textbf{Theorem}}
\newtheorem{lemma}[fed]{\textbf{Lemma}}

\newtheorem{prop}[fed]{\textbf{Proposition}}
\newtheorem{conj}[fed]{\textbf{Conjecture}}
\newtheorem{cor}[fed]{\textbf{Corollary}}
\usepackage{amssymb,bbm,graphicx,epsfig,psfrag,epic,eepic,latexsym}
\usepackage{amsmath}
\usepackage{mathrsfs}
\usepackage{epigraph}

\usepackage{a4wide}
\newcommand{\HH}{\mathcal{H}}
\newcommand{\VV}{\mathcal{V}}

\begin{document}
\title{Bubbles and Onis}
\author{Peter Albers, Urs Frauenfelder}

\address{Peter Albers\\
 Mathematisches Institut\\
 Westf\"alische Wilhelms-Universit\"at M\"unster}
\email{peter.albers@wwu.de}

\address{Urs Frauenfelder\\
 Mathematisches Institut\\
Universit\"at Augsburg}
\email{urs.frauenfelder@math.uni-augsburg.de}

\begin{abstract}
In this article we study the gradient flow equation of a variant of the Rabinowitz action functional on very negative line bundles and relate it to periodic orbits on the base of this bundle. On very negative line bundles there are generically no holomorphic spheres. 
\end{abstract}

\maketitle

\setlength{\epigraphwidth}{5.5cm}

\epigraph{\flushright{\it Dedicated to Paul Rabinowitz}}

\section{Introduction}

Rabinowitz pioneered the application of global methods in Hamiltonian dynamics, see \cite{Rabinowitz_Periodic_solutions_of_Hamiltonian_systems,Rabinowitz_Periodic_solutions_of_Hamiltonian_systems_prescribed}. This motivated Weinstein to formulate his famous conjecture in \cite{Weinstein_The_conjecture} on the existence of periodic solutions of Hamiltonian systems with  fixed energy and arbitrary period. For existence of periodic solutions of fixed period of time-dependent Hamiltonian systems Arnold formulated in the 60s his famous conjecture, see \cite[Chapter 11]{McDuff_Salamon_introduction_symplectic_topology} for a detailed account. Arnold's conjecture motivated Floer to introduce his semi-infinite dimensional Morse homology, nowadays referred to as Floer homology, more precisely as Hamiltonian Floer homology in this context. The analog of Floer homology for the fixed energy problem was introduced by Cieliebak-Frauenfelder in \cite{cieliebak-frauenfelder} and is referred to as \textbf{Rabinowitz Floer homology} in reference to the action functional used in Rabinowitz' fundamental article \cite{Rabinowitz_Periodic_solutions_of_Hamiltonian_systems}. In the present article we study a variant of the Rabinowitz action functional for the fixed period problem. 

Defining Hamiltonian Floer homology on general closed symplectic manifolds $(M,\omega)$ is notoriously difficult due to bubbling-off of holomorphic spheres. Formally bubbling-off of holomorphic spheres is a codimension two phenomenon, however this leads to sever transversality issues which forces the use of multi-valued perturbations and thus rational coefficients, see for example \cite{fukaya-ono, Hofer_polyfolds_and_a_general_Fredholm_theory}. On a sufficiently negative line bundle over the closed symplectic manifold holomorphic spheres generically do not exist. In this article we analyze what will happen to the bubbles. 

We lift the Hamiltonian dynamics from $M$ to the negative line bundle via a modified Rabinowitz action functional. If $(M,\omega)$ is semi-positive one gets a version of Rabinowitz Floer homology on the negative line bundle which is canonically isomorphic to the Floer homology on $M$. In this paper we don't want to assume semi-positivity.

We describe here a moduli space of a new PDE-type problem. We refer to solutions to this problem as onis.\footnote{Oni (\includegraphics[trim = 0 300 0 0, width = 3ex]{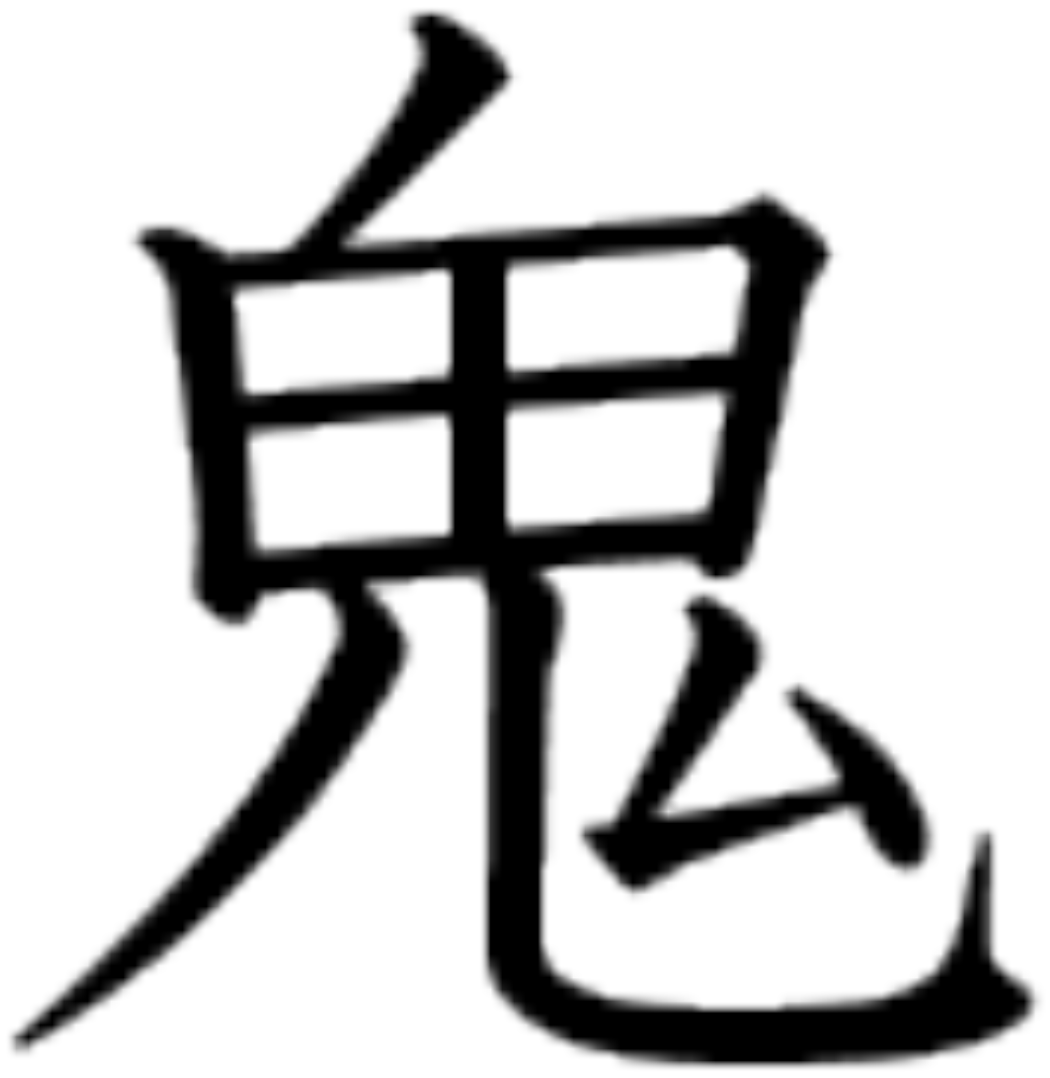}) is Japanese for an imp.} As with holomorphic spheres onis are the obstructions to a well-defined boundary operator in Rabinowitz Floer homology. The advantage of the onis is that for generic almost complex structure one has transversality. In the second part of the paper we explain how this observation is potentially useful. To tame onis we marry them, i.e.~with $\mathbb{Z}/2$-coefficients the oni-obstruction vanishes and one obtains a well-defined boundary operator. So far, it seems that marrying onis is only possible if the Hamiltonian diffeomorphism $\phi$ we are studying admits a square root, i.e.~$\phi=\psi^2$ for some other Hamiltonian $\psi$. Unfortunately, this is not always the case as we showed in \cite{albers-frauenfelder}. Marrying onis apparently comes at a price: we explain that marrying onis in the worst case costs 80\% of the sum of the Betti numbers of $M$. The undaunted reader is cordially invited to read this gloomy story.

\section{From bubbles to Onis}

In Floer homology one considers a closed symplectic manifold
$(M,\omega)$ and a Hamiltonian $H \in C^\infty(M \times S^1)$ which
is periodic in time. For $S^1=\mathbb{R}/\mathbb{Z}$ the circle we
denote by
$$\mathscr{L}_M \subset C^\infty(S^1,M)$$
the component of the free loop space $C^\infty(S^1,M)$ consisting of
contractible loops. We consider the cover
$$\widetilde{\mathscr{L}}_M \to \mathscr{L}_M$$
consisting of equivalence classes $[v,\bar{v}]$ where $v \in
\mathscr{L}_M$ and $\bar{v}$ is a filling disk of $v$. Here two
pairs are equivalent if the loops coincide and the integrals of
$\omega$ as well as of one and hence every representative of the
first Chern class $c_1(TM)$ agree on the filling disk. Therefore the
group of decktransformations of the cover $\widetilde{\mathscr{L}}_M
\to \mathscr{L}_M$ is the group
$$\Gamma=\frac{\pi_2(M)}{\mathrm{ker}\omega \cap
\mathrm{\ker}(c_1(TM))}.$$ The action functional of classical
mechanics
$$\mathcal{A}_H \colon \widetilde{\mathscr{L}}_M \to \mathbb{R}$$
is given by
$$\mathcal{A}_H\big([v,\bar{v}]\big)=-\int \bar{v}^*\omega-
\int_0^1 H(v(t),t)dt.$$ If $[v,\bar{v}]$ is a critical point of
$\mathcal{A}_H$, then $v$ is a contractible time-one periodic orbit
of the Hamiltonian vector field of $H$, i.e. a solution of the ODE
$$\partial_t v(t)=X_{H_t}(v(t)), \quad t \in S^1,$$
where the Hamiltonian vector field is defined by
$dH_t=\omega(X_{H_t},\cdot)$, with $H_t=H(\cdot,t) \in C^\infty(M)$
for $t \in S^1$. The Hamiltonian $H$ is called nondegenerate if for
each contractible periodic orbit $v$ the time-one map $\phi_H$ of
the Hamiltonian vector field of $H$ satisfies
$$\mathrm{det}\big(d\phi(v(0))-\mathrm{id}|_{T_{v(0)}M}\big) \neq 0.$$
If the Hamiltonian is nondegenerate there are only finitely many
$\Gamma$-orbits of critical points of $\mathcal{A}_H$ and the Floer
chain space $CF_*(H)$ can be defined as the $\mathbb{Z}_2$-vector
space consisting of infinite sums
$$\xi=\sum_{c \in \mathrm{crit}(\mathcal{A}_H)} \xi_c c$$
with coefficients $\xi_c \in \mathbb{Z}_2$ which satisfy for any $r
\in \mathbb{R}$ the finiteness condition
$$\#\{c \in \mathrm{crit}(\mathcal{A}_H): \xi_c \neq 0,\,\,
\mathcal{A}_H(c)>r\}<\infty.$$ The grading on $CF_*(H)$ is given by
the Conley-Zehnder index. Using the action of $\Gamma$ on
$\mathrm{crit}(\mathcal{A}_H)$ the Floer chain space can be endowed
with the structure of a module over the Novikov ring of $\Gamma$. As a matter of fact the Novikov ring over a field is itself a field.

To define a boundary operator on $CF_*(H)$ Floer considers the
$L^2$-gradient flow equation for the action functional of classical
mechanics. We denote by $\mathcal{J}$ circle families $J_t$ for $t
\in S^1$ of $\omega$-compatible almost complex structures on $M$.
For $J \in \mathcal{J}$ the metric $\mathfrak{m}_J$ on
$\widetilde{\mathscr{L}}_M$ at a point $[v,\bar{v}]$ is given for
two tangent vectors $\hat{v}_1,\hat{v}_2 \in T_{[v,\bar{v}]}
\widetilde{\mathscr{L}}_M=\Gamma(S^1,v^* TM)$ by
$$\mathfrak{m}_J(\hat{v}_1,\hat{v}_2)=\int_0^1 \omega\big(\hat{v}_1(t)
,J_t(v(t))\hat{v}_2(t)\big)dt.$$ The gradient $\nabla_J
\mathcal{A}_H$ of the action functional $\mathcal{A}_H$ with respect
to the metric $\mathfrak{m}_J$ at a point $[v,\bar{v}] \in
\widetilde{\mathscr{L}}_M$ is given by
$$\nabla_J \mathcal{A}_H([v,\bar{v}])=J(v)(\partial_t v-X_H(v)).$$
A gradient flow line $w=[v,\bar{v}] \in
C^\infty(\mathbb{R},\widetilde{\mathscr{L}}_M)$ is formally a
solution of the ODE
$$\partial_s w(s)+\nabla_J \mathcal{A}_H(w(s))=0, \quad s \in
\mathbb{R}$$ and therefore $v \in C^\infty(\mathbb{R} \times S^1,M)$
is a solution of the PDE
$$\partial_s v+J(v)(\partial_t v-X_H(v))=0$$
which is a perturbed holomorphic curve equation. For critical points
$c_-,c_+ \in \mathrm{crit}(\mathcal{A}_H)$ we denote by
$\mathcal{M}(c_-,c_+;J)$ the moduli space of unparametrised gradient
flow lines $[w]$ of $\mathcal{A}_H$ which asymptotically satisfy
$\lim_{s \to \pm \infty}w(s)=c_\pm$. By
$$\mathcal{J}_{\textrm{reg}}(H) \subset \mathcal{J}$$
we denote the subset of second category of $J \in \mathcal{J}$ such
that the linearization of Floer's gradient flow equation with
respect to $J_t$ along any finite energy gradient flow line is
surjective. If $J \in \mathcal{J}_{\mathrm{reg}}$ then the moduli
space of gradient flow lines is a smooth manifold of dimension
$$\mathrm{dim}\mathcal{M}(c_-,c_+;J)=\mu_{CZ}(c_-)-\mu_{CZ}(c_+)-1.$$
In particular if $\mu_{CZ}(c_-)=\mu_{CZ}(c_+)+1$ then the moduli
space is zero dimensional. Unfortunately without additional
assumptions on $(M,\omega)$ like monotonicity or more generally
semipositivity there is little hope that this moduli space is also
compact. This is due to bubbling of holomorphic spheres. If it is
compact then it is a finite set of points and one sets
$$n(c_-,c_+,J)=\#_2\mathcal{M}(c_-,c_+;J)$$
where $\#_2$ denotes cardinality modulo two. In this case the Floer
boundary map
$$\partial \colon CF_*(H) \to CF_{*-1}(H)$$
is defined for $\xi=\sum_{c \in \mathrm{crit}(\mathcal{A}_H)}\xi_c c
\in CF_*(H)$ by
$$\partial(\xi)=\sum_{c' \in \mathrm{crit}(\mathcal{A}_H)}
\sum_{c \in \mathrm{crit}(\mathcal{A}_H)} \xi_c n(c,c';J) c'.$$ If
the moduli spaces are not compact then one way to still define the
Floer boundary operator is by using abstract perturbation theory. In
this case one first compactifies the moduli space via bubbles and
then abstractly perturbs it. Because of the possibility that the
bubbles have nontrivial automorphism group this perturbation is
multivalued and one has to restrict in this case to rational
coefficients.
\\ \\
In this paper we study a different kind of perturbation of the
moduli spaces of Floer's gradient flow equation which also leads to
compact moduli spaces. We have to assume in addition that
$(M,\omega)$ satisfies the Bohr-Sommerfeld condition, i.e.
$[\omega]$ lies in the image of $H^2(M;\mathbb{Z})\to H^2_{dR}(M)$.
Under this assumption there exists a hermitian line bundle
$$E_\omega \to M$$
whose first Chern class satisfies
$$c_1(E_\omega)=-[\omega].$$
For $\nu \in \mathbb{N}$ we consider the hermitian line bundle
$$E=E^\nu_\omega=E_\omega^{\otimes \nu}$$
whose first Chern class satisfies
$$c_1(E)=-\nu[\omega].$$
We fix a hermitian connection $\alpha$ on $E$ whose curvature
satisfies
$$F_\alpha=\nu \omega.$$
If $p \colon E \to M$ denotes the canonical projection, we endow $E$
with the symplectic form
$$\omega_E=d(\pi |u|^2 \alpha)+p^* \omega.$$
Thinking of $M$ as the zero section in $E$ the restriction of
$\omega_E$ to $M$ coincides with $\omega$. However, the virtual
dimension of the moduli space of holomorphic curves drops in $E$,
because the line bundle is negative. If the line bundle is negative
enough generically there are no holomorphic curves left in $E$. To
have a quantitative statement to say what "negative enough" means we
consider the Auroux constant. If $\beta \in \Omega^2(M)$ and $J$ is
an $\omega$-compatible almost complex structure we set
$$\gamma_{\beta,J}=\beta(\cdot,J\cdot) \in \Gamma(T^*M \oplus
T^*M)$$ and abbreviate
$$\kappa_\beta(J)=||\gamma_{\beta,J}||_J$$
where the notation $||\cdot||_J$ means that we take the norm with
respect to the metric $\omega(\cdot,J\cdot)$. We then set
$$\kappa(J)=\inf\big\{\kappa_\beta(J):
d\beta=0,\,\,[\beta]=c_1(TM)\big\}$$ and finally
$$\kappa(\omega)=\inf\big\{\kappa(J):\,J\,\,\,\omega\textrm{-compatible}\big\}.$$
We assume in the following that $\nu \in \mathbb{N}$ satisfies
$$\nu >\max\big\{n+\kappa(\omega)-2,\kappa(\omega)\big\}.$$
This condition turns out to be sufficient to make sure that there
are $\omega$-compatible almost complex structures for which no
holomorphic spheres exist.

We denote by $\mu \in C^\infty(E)$ the function
$$\mu(u)=\pi(|u|^2-1), \quad u \in E.$$
Denote by $\mathscr{L}_E$ the component of contractible loops in the
free loop space $C^\infty(S^1,E)$ and by $\widetilde{\mathscr{L}}_E$
the cover of $\mathscr{L}_E$ consisting of equivalence classes of
pairs $[u,\bar{u}]$ where $u \in \mathscr{L}_E$, $\bar{u} \in
C^\infty(D,E)$ with $D$ the unit disk is a filling disk of $u$, and
two filling disks are equivalent if $\omega_E$ and each
representative of $p^*c_1(TM)$ agrees on them. For $H \in
C^\infty(M\times S^1)$ the time-dependent Hamiltonian on the base
$M$, we denote for $t \in S^1$
$$\widehat{H}_t(u)=\nu\pi|u|^2 H_t(p(u)), \quad u \in E$$
the fiberwise quadratic lift of $H$ to $E$. We consider the
following variant of Rabinowitz action functional
$$\mathcal{A}^\mu_H \colon \widetilde{\mathscr{L}}_E
\times \mathbb{R} \to \mathbb{R}$$ given for
$\big([u,\bar{u}],\eta\big) \in \widetilde{\mathscr{L}}_E \times
\mathbb{R}$ by
$$\mathcal{A}^\mu_H\big([u,\bar{u}],\eta\big)
=-\int_D \bar{u}^*\omega_E-\int_0^1\widehat{H}_t(u)dt -\eta\int_0^1
\mu(u)dt.$$ If $P \colon \widetilde{\mathscr{L}}_E \times \mathbb{R}
\to \widetilde{\mathscr{L}}_M$ denotes the projection induced from
the projection $p \colon E \to M$, then
$$P\big(\mathrm{crit}(\mathcal{A}_H^\mu)\big)=
\mathrm{crit}(\mathcal{A}_H).$$ However, the correspondence is not
one to one, but for each $c \in \mathrm{crit}(\mathcal{A}_H)$ there
is a whole $\mathbb{Z}\times S^1$- family of critical points of
$\mathrm{crit}(\mathcal{A}_H^\mu)$, i.e.
$$P|_{\mathrm{crit}(\mathcal{A}_H^\mu)} ^{-1}(c) \cong \mathbb{Z}
\times S^1, \quad c \in \mathrm{crit}(\mathcal{A}_H).$$ The
geometric origin of this fact is the following. The circle acts on
$E$ by
$$r*u=e^{-2\pi ir}u, \quad r \in S^1=\mathbb{R}/\mathbb{Z},\,\,
u \in E.$$ In fact this action is Hamiltonian with respect to
$\omega_E$ with moment map $\mu$. This action gives rise to a
$\mathbb{Z}\times S^1$-action on $\mathscr{L}_E$ which is given for
$u \in \mathscr{L}_E$ by
$$\big((n,r)_*u\big)(t)=(nt+r)_*u(t), \,\,t \in S^1,\,\,(n,r)
\in \mathbb{Z}\times S^1.$$ We lift this action to the cover
$\widetilde{\mathscr{L}}_E \to \mathscr{L}_E$ and extend it
trivially to $\widetilde{\mathscr{L}}_E \times \mathbb{R}$. The
differential $d\mathcal{A}_H^\mu$ is invariant under this action and
hence we get a $\mathbb{Z} \times S^1$-action on
$\mathrm{crit}(\mathcal{A}^\mu_H)$. The projection $P$ induces a
bijection
$$\frac{\mathrm{crit}(\mathcal{A}_H^\mu)}{\mathbb{Z}\times S^1}
\cong \mathrm{crit}(\mathcal{A}_H).$$ A section for the
$\mathbb{Z}$-action on $\mathrm{crit}(\mathcal{A}^\mu_H)$ is given
by the \emph{winding number}
$$\mathfrak{w} \colon \mathrm{crit}(\mathcal{A}^\mu_H) \to
\mathbb{Z}$$ which for a critical point $([u,\bar{u}],\eta)$ is
given by
$$\mathfrak{w}\big([u,\bar{u}],\eta\big)=\int u^*\alpha-
\nu \int \bar{u}^*p^*\omega.$$
 The connection $\alpha$ induces a
splitting
\begin{equation}\label{split}
TE=\VV \oplus \HH
\end{equation}
into vertical and horizontal subbundles. Recalling that $p \colon E
\to M$ denotes the canonical projection, then for each $e \in E$ we
have canonical identifications
$$\HH_e=T_{p(e)}M, \quad \VV_e=E_{p(e)}.$$
We denote by $I$ the complex structure in $\VV$ coming from the
complex structure of the hermitian vector bundle $E \to M$. For $J
\in \mathcal{J}$ we denote by abuse of notation its lift to a family
of complex structure on $\HH$ also by $J$. We extend the vector bundle
$\mathrm{End}(\HH,\VV)$ over $E$ trivially to $E \times S^1$. For a
section $B \in \Gamma_0\big(E \times S^1,\mathrm{End}(\HH,\VV)\big)$ we
abbreviate
$$B_t=B( \cdot, t) \in \Gamma_0\big(E,\mathrm{End}(\HH,\VV)\big), \quad
t \in S^1.$$ 
Since $H$ is non-degenerate $X_H$ has only finitely many periodic orbits. From now on we fix around each periodic orbit a neighborhood which contracts onto the orbit. Furthermore, we assume that all these neighborhoods are disjoint. Let $\mathcal{U}$ be the union of these neighborhoods. If $\Gamma_0$ stands for sections with compact support
we introduce the vector space
\begin{equation}\label{eqn:space_of_Bs}\mathfrak{B}(J)=\Big\{B \in
\Gamma_0\big(E \times S^1,\mathrm{End}(\HH,\VV)\big): B_tJ_t=-IB_t,\,\,t
\in S^1\text{ and }B_t(e)=0\;\forall e\text{ with }p(e)\in\mathcal{U}\Big\}.
\end{equation}
Even though $\mathfrak{B}(J)$ depends on $\mathcal{U}$ we will suppress this in the notation.

With respect to the splitting (\ref{split}) we
introduce for $B \in \mathfrak{B}(J)$ the following circle family of
almost complex structures on $E$
$$J_t^B=\left(\begin{array}{cc}
I & B_t \\
0 & J_t
\end{array}\right).$$
If $B$ is different from zero then $J^B$ is not
$\omega_E$-compatible, however if $B$ is a small enough perturbation
then $J^B$ is still $\omega_E$-tame. We therefore introduce the
nonempty open convex subset
$$\mathfrak{B}^T(J) \subset \mathfrak{B}(J)$$
consisting of all $B \in \mathfrak{B}(J)$ such that $J_t^B$ is
$\omega_E$-tame for any $t \in S^1$. If $J_t$ is a smooth family of
$\omega$-compatible almost complex structures and $B_t$ is a smooth
family of compactly supported sections from $E$ to
$\mathrm{End}(\HH,\VV)$ such that $B_t \in \mathfrak{B}(J_t)$ for every
$t \in S^1$ we denote by $\mathfrak{m}_B$ the bilinear form on
$T\big(\widetilde{\mathscr{L}}_E \times \mathbb{R}\big)$ which is
given for $\big([u,\bar{u}],\eta\big) \in \widetilde{\mathscr{L}}_E
\times \mathbb{R}$ and $(\hat{u}_1,\hat{\eta}_1),
(\hat{u}_2,\hat{\eta}_2) \in T_{([u,\bar{u}],\eta)}\big(
\widetilde{\mathscr{L}}_E \times \mathbb{R}\big) =\Gamma(S^1,u^*TE)
\times \mathbb{R}$ by the formula
$$\mathfrak{m}_B\big((\hat{u}_1,\hat{\eta}_1),
(\hat{u}_2,\hat{\eta}_2)\big)= -\int_0^1\omega_E\big(
J^{B_t}_t(u(t))\hat{u}_1(t),\hat{u}_2(t)\big)dt+ \hat{\eta}_1 \cdot
\hat{\eta}_2.$$ If $B$ is different from zero, the bilinear form
$\mathfrak{m}_B$ is not symmetric. However, it is nondegenerate, and
if $B$ is small it is positive.
 Denote by
$$R=X_{\mu} \in \Gamma(TE)$$
the infinitesimal generator of the $S^1$-action on $E$. The gradient
of Rabinowitz action functional with respect to the bilinear form
$\mathfrak{m}_B$ at $w=\big([u,\bar{u}],\eta\big) \in
\widetilde{\mathscr{L}}_E \times \mathbb{R}$ defined implicitly by
the condition
$$d\mathcal{A}^\mu_H(w)=\mathfrak{m}_B\big(\nabla_B
\mathcal{A}^\mu_H(w),\hat{w}\big), \quad \forall\,\, \hat{w} \in
T_w\big(\widetilde{\mathcal{L}}_E \times \mathbb{R}\big)$$ is given
by
$$\nabla_B\mathcal{A}^\mu_H
\big([u,\bar{u}],\eta\big)
=\left(\begin{array}{c}
J^B\big(\partial_t u-X_{\widehat{H}}(u)-\eta R(u)\big)\\
-\int_0^1\mu(u)dt
\end{array}\right).$$
We point out that $\nabla_B\mathcal{A}^\mu_H$ is not an honest gradient since $\mathfrak{m}_B$ is not symmetric. But
$$
d\mathcal{A}^\mu_H(w)\nabla_B
\mathcal{A}^\mu_H(w)=\mathfrak{m}_B\big(\nabla_B
\mathcal{A}^\mu_H(w),\nabla_B
\mathcal{A}^\mu_H(w)\big)>0
$$
away from the critical points. Moreover $\mathfrak{m}_B$ is a honest metric on loops contained in $p^{-1}(\mathcal{U})$. Thus, the vector field $\nabla_B\mathcal{A}^\mu_H(w)$ is a pseudo-gradient for $\mathcal{A}^\mu_H$.

A gradient flow line $w \in C^\infty(\mathbb{R},
\widetilde{\mathscr{L}}_E \times \mathbb{R})$ of Rabinowitz action
functional with respect to $\mathfrak{m}_B$ formally satisfies
$$\partial_s w(s)+
\nabla_B\mathcal{A}^\mu_H(w(s))=0, \quad s \in \mathbb{R}.$$ Hence
$(u,\eta) \in C^\infty(\mathbb{R}\times S^1,E) \times
C^\infty(\mathbb{R},\mathbb{R})$ is a solution of the problem
\begin{equation}\label{gr1}
\left. \begin{array}{c}
\partial_s u+J^B(u)\big(\partial_t u-X_{\widehat{H}}(u)-\eta
R(u)\big)=0\\
\partial_s \eta-\int_0^1 \mu(u)dt=0.
\end{array}\right\}
\end{equation}
Writing
$$\partial_s u=\partial^v_s u+\partial^h_s u, \quad
\partial_t u=\partial^v_t u+\partial^h_t u$$
with respect to the splitting (\ref{split})
we can rewrite (\ref{gr1}) to
\begin{equation}\label{gr2}
\left.\begin{array}{c}
\partial_s^v u+I\Big(\partial_t^v u-\big(\nu H(p(u))+\eta
\big)R(u)\Big)
+B\Big(\partial_t^h(u)-X_H(p(u))\Big)=0\\
\partial_s^h u+J(p(u))\Big(\partial_t^h(u)-X_H(p(u))\Big)=0\\
\partial_s \eta-\int_0^1 \mu(u)dt=0.
\end{array}\right\}
\end{equation}
We abbreviate by $\mathcal{R}(B)$ the moduli space of all
unparametrised flow lines $[w]$ of $\nabla_B \mathcal{A}^\mu_H$ such
that $\lim_{s \to \pm \infty}w(s) \in
\mathrm{crit}(\mathcal{A}^\mu_H)$ exists. We further introduce the
evaluation maps
$$\mathrm{ev}_\pm \colon \mathcal{R}(B) \to
\mathrm{crit}(\mathcal{A}^\mu_H)$$ defined by
$$\mathrm{ev}_\pm([w])=\lim_{s \to \pm \infty}w(s).$$
We denote by $\mathcal{J}^\nu \subset \mathcal{J}$ the nonempty open
subset of $J \in \mathcal{J}$ such that
$$\nu >\max\big\{n+\kappa(J_t)-2,\kappa(J_t)\big\}, \quad
t \in S^1.$$ and set
$$\mathcal{J}^\nu_{\textrm{reg}}(H)=\mathcal{J}^\nu \cap
\mathcal{J}_{\textrm{reg}}(H).$$ For $J \in \mathcal{J}^\nu$ we
abbreviate by
$$\mathfrak{B}_{\textrm{reg}}(J) \subset \mathfrak{B}(J)$$
the subset of perturbations $B \in \mathfrak{B}(J)$ which satisfy
the following three conditions.
\begin{description}
 \item[(i)] The linearization of the flow equation for
 $\nabla_B \mathcal{A}^\mu_H$ on $\mathcal{R}(B)$ is surjective.
 \item[(ii)] The evaluation maps $\mathrm{ev}_+$ and $\mathrm{ev}_-$
 are transverse to each other.
 \item[(iii)] For each $t \in S^1$ there are no nonconstant
 $J^B_t$-holomorphic spheres on $E$.
\end{description}
\begin{prop}\label{generic}
For $J \in \mathcal{J}^\nu_{\mathrm{reg}}$ the subset
$\mathfrak{B}_{\mathrm{reg}}(J) \subset \mathfrak{B}(J)$ is of
second category.
\end{prop}
Recall that $\mathfrak{B}^T(J)$ denotes the set of perturbations $B
\in \mathfrak{B}(J)$ such that $J^B$ is $\omega_E$-tame. We set
$$\mathfrak{B}^T_{\mathrm{reg}}(J)=\mathfrak{B}^T(J)
\cap \mathfrak{B}_{\mathrm{reg}}(J).$$ We make the following
definition. A \emph{lift}
$$\ell \colon \mathrm{crit}(\mathcal{A}_H) \to
\mathrm{crit}(\mathcal{A}^\mu_H)$$ is a section for the projection
$P \colon \mathrm{crit}(\mathcal{A}^\mu_H) \to
\mathrm{crit}(\mathcal{A}_H)$, i.e. $P \circ \ell=
\mathrm{id}|_{\mathrm{crit}(\mathcal{A}_H)}$, satisfying
$$\mathfrak{w}(\ell(c))=0, \quad c \in \mathrm{crit}(\mathcal{A}_H).$$
Note that for each $c \in \mathrm{crit}(\mathcal{A}_H)$ there is an
$S^1$-ambiguity for the choice of $\ell(c)$. For a lift $\ell$ we
introduce the following subset of $\mathcal{R}(B)$
$$\mathcal{R}(B,\ell)=\big\{[w] \in \mathcal{R}(B):
\mathrm{ev}_-([w]) \in \ell(\mathrm{crit}(\mathcal{A}_H))\big\}.$$
We say that a lift is $B$-\emph{admissible} if the following two
conditions hold.
\begin{description}
 \item[(i)] The restriction of the linearization of the
 flow equation of $\nabla_B \mathcal{A}^\mu_H$ to
 $\mathcal{R}(B,\ell)$ is surjective.
 \item[(ii)] The evaluation map
 $\mathrm{ev}_+|_{\mathcal{R}(B,\ell)}$ is transverse to
 $\mathrm{ev}_-$.
\end{description}
If $B \in \mathfrak{B}_{\mathrm{reg}}(J)$, then the evaluation maps
$\mathrm{ev}_-$ and $\mathrm{ev}_+$ are transverse to each other by
assumption and hence a generic lift $\ell$ is $B$-admissible. For
$c_-, c_+ \in \mathrm{crit}(\mathcal{A}_H)$, $J \in \mathcal{J}$, $B
\in \mathfrak{B}(J)$ and a lift $\ell$ we abbreviate
$$\mathcal{R}(c_-,c_+;B,\ell)=\big\{[w]
\in \mathcal{R}(B): \mathrm{ev}_-([w])=\ell(c_-),\,\,
\mathrm{ev}_+([w])\in S^1\cdot \ell(c_+)\big\}.$$ The condition for the
positive asymptotic can be rephrased by saying that
$\mathrm{ev}_+([w])$ is a critical point of $\mathcal{A}^\mu_H$ of
winding number zero. Hence the condition for the positive asymptotic
is actually independent of the choice of the lift. If $B \in
\mathfrak{B}_{\mathrm{reg}}(J)$, $\ell$ is a $B$-admissible lift,
and $c_- \neq c_+ \in \mathrm{crit}(\mathcal{A}_H)$, then the moduli
space $\mathcal{R}(c_-,c_+;B,\ell)$ is a smooth manifold of
dimension
$$\mathrm{dim}\mathcal{R}(c_-,c_+;B,\ell)=\mu_{CZ}(c_-)-
\mu_{CZ}(c_+)-1$$ and therefore coincides with the exspected
dimension of the moduli space of unparametrized Floer gradient flow
lines from $c_-$ to $c_+$.
\begin{thm}\label{finite}
Assume that the Conley-Zehnder indices of $c_-, c_+ \in
\mathrm{crit}(\mathcal{A}_H)$ satisfy
$\mu_{CZ}(c_-)=\mu_{CZ}(c_+)+1$, that $J \in \mathcal{J}^\nu$, $B
\in \mathfrak{B}^T_{\mathrm{reg}}(J)$, and $\ell$ is a
$B$-admissible lift, then the moduli space
$\mathcal{R}(\ell_{c_-},c_+;B)$ is a finite set.
\end{thm}
Under the assumptions of the theorem, we can now set
$$\varrho(c_-,c_+;B,\ell)=\#_2 \mathcal{R}(c_-,c_+;B,\ell)$$
the cardinality modulo two of the moduli spaces. If the
Conley-Zehnder indices do not satisfy
$\mu_{CZ}(c_-)-\mu_{CZ}(c_+)=1$, then we set
$\varrho(c_-,c_+;B,\ell)=0$.
\begin{prop}\label{independent}
If $J \in \mathcal{J}^\nu_{\mathrm{reg}}$, $B \in
\mathfrak{B}^T_{\mathrm{reg}}(J)$ and $\ell$ is a $B$-admissible
lift, then the numbers $\varrho(c_-,c_+;B,\ell)$ are independent of
the choice of $\ell$.
\end{prop}
Hence under the assumptions of the proposition we are allowed to set
$$\varrho(c_-,c_+;B)=\varrho(c_-,c_+;B,\ell)$$
for any $B$-admissible lift $\ell$. The following Lemma is an
immediate consequence of Gromov compactness applied to the shadows
$P(w)$ of flow lines of $\nabla_B \mathcal{A}_H^\mu$ by noting that
these are flow lines of $\nabla_J \mathcal{A}_H$.
\begin{lemma}
If $J \in \mathcal{J}^\nu_{\mathrm{reg}}$, and $B \in
\mathfrak{B}^T_{\mathrm{reg}}(J)$, then for each real number $r \in
\mathbb{R}$ and $c_-\in\mathrm{crit}(\mathcal{A}_H)$ the set $\big\{c_+ \in \mathrm{crit}(\mathcal{A}_H):
\mathcal{A}_H(c_+)>r, \,\, \varrho(c_-,c_+;B) \neq 0\big\}$ is finite.
\end{lemma}
In view of the Lemma we can define a $\mathbb{Z}_2$-linear map
$$\mathfrak{R} \colon CF_*(H) \to CF_{*-1}(H)$$
which is given for $\xi=\sum_{c \in \mathrm{crit}(\mathcal{A}_H)}
\xi_c c \in CF_*(H)$ by
$$\mathfrak{R}(\xi)=\sum_{c' \in \mathrm{crit}(\mathcal{A}_H)}
\sum_{c \in \mathrm{crit}(\mathcal{A}_H)} \xi_c \varrho(c,c';B)
c'.$$ We refer to $\mathfrak{R}$ as the Rabinowitz Floer map. Since
the winding number is unchanged under the action of
$\Gamma_0=\pi_2(M)/{\mathrm{ker}(\omega)}=\pi_2(E)/{\mathrm{ker}(\omega_E)}$
on $\widetilde{\mathscr{L}}_E$ we conclude that the Rabinowitz Floer
map is linear with respect the action of the group ring
$\mathbb{Z}_2[\Gamma_0]$ on $CF_*(H)$.
\\ \\
We next describe the square of the Rabinowitz Floer map. For this we
need the following Definition.
\begin{fed}
Assume that $J \in \mathcal{J}$ and $B \in \mathfrak{B}(J)$. An
\emph{Oni} is an element $[w] \in \mathcal{R}(B)$ whose asymptotics
satisfy $\mathfrak{w}(\mathrm{ev}_-(w))=1$ and
$\mathfrak{w}(\mathrm{ev}_+(w))=0$.
\end{fed}
We denote by $\mathcal{O}(B)$ the moduli space of Onis. If $B=0 \in
\mathfrak{B}(J)$ then $M$ interpreted as zero section in $E$ is a
complex submanifold of $J^0$. Therefore by positivity of
intersections Onis cannot exist and we have
$$\mathcal{O}(0)=\emptyset.$$
But without perturbation all bubbles in $M$ survive. In view of this
the following statement makes sense: \emph{Onis are born out of
bubbles.}
\\ \\
If $c_-, c_+ \in \mathrm{crit}(\mathcal{A}_H)$ we abbreviate
$$\mathcal{O}(c_-,c_+;B)=\big\{[w] \in \mathcal{O}(B):
P(\mathrm{ev}_\pm(w))=c_\pm\big\}.$$ If $B \in
\mathfrak{B}_{\mathrm{reg}}(J)$ the moduli space of Onis is a smooth
manifold and its dimension is given by
$$\mathrm{dim}\mathcal{O}(c_-,c_+;B)=\mu_{CZ}(c_-)-\mu_{CZ}(c_+)-2.$$
If $J \in \mathcal{J}^\nu$, $B \in
\mathfrak{B}^T_{\mathrm{reg}}(J)$, and $c_-,c_+ \in
\mathrm{crit}(\mathcal{A}_H)$ satisfy $\mu(c_-)-\mu(c_+)=2$, then
the same compactness arguments which lead to Theorem~\ref{finite}
also show that the moduli space of Onis $\mathcal{O}(c_-,c_+;B)$ is
a finite set. Hence in this situation we set
$$\varpi(c_-,c_+;B)=\#_2 \mathcal{O}(c_-,c_+;B)$$
and define a map
$$\mathfrak{O} \colon CF_*(H) \to CF_{*-2}(H)$$
which is given for $\xi=\sum_{c \in \mathrm{crit}(\mathcal{A}_H)}
\xi_c c \in  CF_*(H)$ by
$$\mathfrak{O}(\xi)=\sum_{c' \in \mathrm{crit}(\mathcal{A}_H)}
\sum_{c \in \mathrm{crit}(\mathcal{A}_H)} \xi_c \varpi(c,c';B) c'.$$
We refer to the map $\mathfrak{O}$ as the \emph{Oni-map}. The
Oni-map is the obstruction for the Rabinowitz Floer map to be a
boundary operator as the following theorem shows.
\begin{thm}\label{oni}
Assume that $J \in \mathcal{J}^\nu_{\mathrm{reg}}$, $B \in
\mathfrak{B}^T_{\mathrm{reg}}(J)$, then
$\mathfrak{R}^2=\mathfrak{O}$.
\end{thm}

\section{Proofs}

\subsection{Proof of Proposition~\ref{independent}}

Given two $B$-admissible lifts $\ell_0,\ell_1 \colon
\mathrm{crit}(\mathcal{A}_H) \to \mathrm{crit}(\mathcal{A}^\mu_H)$
and two critical points $c_-,c_+ \in \mathrm{crit}(\mathcal{A}_H)$
satisfying $\mu_{CZ}(c_-)=\mu_{CZ}(c_+)+1$ we have to show that
\begin{equation}\label{indep}
\varrho(c_-,c_+;B,\ell_0)=\varrho(c_-,c_+;B,\ell_1).
\end{equation}
Choose a smooth family $L=\{\ell_r\}_{r \in [0,1]}$ of lifts $\ell_r
\colon \mathrm{crit}(\mathcal{A}_H) \to
\mathrm{crit}(\mathcal{A}^\mu_H)$ interpolating between $\ell_0$ and
$\ell_1$. We consider the moduli space
$$\mathcal{R}(B,L)=\big\{(r,[w]): r \in [0,1],\,\,[w]\in
\mathcal{R}(B,\ell_r)\big\}.$$ The boundary of this moduli space is
given by
$$\partial\mathcal{R}(B,L)=\mathcal{R}(B,\ell_0) \sqcup
\mathcal{R}(B,\ell_1).$$ Therefore to prove (\ref{indep}) it
suffices to show that $\mathcal{R}(B,L)$ is compact. In view of
compactness for the gradient flow equation of Rabinowitz action
functional, see Theorem~\ref{compact}, the only obstruction to
compactness is breaking of gradient flow lines. Since $B$ is
regular, and we consider a 1-dimensional moduli problem it remains
to rule out two times broken flow lines $[w^1]\#[w^2]$ of $\nabla_B
\mathcal{A}^\mu_H$ for which there exists $r \in [0,1]$ such that
\begin{equation}\label{asy1}
\left.\begin{array}{c}
\mathrm{ev}_-([w^1])=\ell_r(c_-),\\
\mathrm{ev}_+([w^1])=\mathrm{ev}_-([w^2]),\\
\mathfrak{w}\big(\mathrm{ev}_+[w^2])\big)=0,\\
P\mathrm{ev}_+[w^2]=c_+.
\end{array}\right\}.
\end{equation}
Its shadow $[Pw^1] \# [Pw^2]$ is then a broken flow line of
$\nabla_J\mathcal{A}_H$ satisfying
\begin{equation}\label{asy2}
\left.\begin{array}{c}
\mathrm{ev}_-([Pw^1])=c_-,\\
\mathrm{ev}_+([Pw^1])=\mathrm{ev}_-([Pw^2]),\\
\mathrm{ev}_+[Pw^2]=c_+.
\end{array}\right\}.
\end{equation}
Since $J$ is regular, we conclude that either $Pw^1$ or $Pw^2$ has
to be constant. We first rule out the case that $Pw^1$ is constant.
If this case occured, the positive asymptotic of the first flow line
would satisfy
$$P\mathrm{ev}_+([w^1])=c_-, \quad
\mathfrak{w}(\mathrm{ev}_+[w^1])>0.$$ But then $w^2$ belongs to a
moduli space of negative virtual dimension which contradicts the
assumption that $B$ is regular. Hence the case that $Pw^1$ is
constant does not occur. If $Pw^2$ is constant, the negative
asymptotic of the second flow line meets the condition
$$P\mathrm{ev}_-([w^2])=c_+, \quad
\mathfrak{w}(\mathrm{ev}_-[w^2])<0$$ implying that $w^1$ lies in a
moduli space of negative virtual dimension. Again this contradicts
the regularity of the perturbation $B$. Hence no breaking occurs and
the Proposition is proved. \hfill $\square$

\subsection{Proof of Theorem~\ref{oni}}

We pick $c_-,c_+ \in \mathrm{crit}(\mathcal{A}_H)$ satisfying
$\mu_{CZ}(c_-)=\mu_{CZ}(c_+)+2$. We have to prove
\begin{equation}\label{onieq}
\sum_{c \in \mathrm{crit}(\mathcal{A}_H)}\varrho(c_-,c;B)
\varrho(c,c_+;B)=\varpi(c_-,c_+;B).
\end{equation}
By Proposition~\ref{independent} the number
$\varrho(c,c_+;B)=\varrho(c,c_+;B,\ell)$ is independent of the
choice of the $B$-admissible lift $\ell$. Therefore for a
$B$-admissible lift $\ell$ the left hand side of (\ref{onieq}) can
be interpreted as the modulo two number of unparametrized
nonconstant broken flow lines $[w^1]\#[w^2]$ of $\nabla_B
\mathcal{A}_H^\mu$ subject to the following asymptotic conditions
\begin{equation}\label{as1}
\left.\begin{array}{c}
\mathrm{ev}_-([w^1])=\ell(c_-),\\
\mathrm{ev}_+([w^1])=\mathrm{ev}_-([w^2]),\\
\mathfrak{w}\big(\mathrm{ev}_+([w^1])\big)=0,\\
\mathfrak{w}\big(\mathrm{ev}_+[w^2])\big)=0,\\
P\mathrm{ev}_+[w^2]=c_+.
\end{array}\right\}
\end{equation}
Let us consider the one dimensional moduli space
$\mathcal{R}(\ell_{c_-},c_+;B)$ of all unparametrized flow lines
from $\nabla_B \mathcal{A}_H^\mu$ from $\ell_{c_-}$ to a point in
$S^1 \ell_{c_+}$. By compactness for gradient flow lines of
$\nabla_B \mathcal{A}_H^\mu$ we conclude that
$\mathcal{R}(\ell_{c_-},c_+;B)$ can be compactified to a one
dimensional manifold with boundary whose boundary points are
unparametrized nonconstant broken flow lines $[w^1]\#[w^2]$ of
$\nabla_B \mathcal{A}_H^\mu$ which meet the asymptotic conditions
$$\left.\begin{array}{c}
\mathrm{ev}_-([w^1])=\ell(c_-),\\
\mathrm{ev}_+([w^1])=\mathrm{ev}_-([w^2]),\\
\mathfrak{w}\big(\mathrm{ev}_+[w^2])\big)=0,\\
P\mathrm{ev}_+[w^2]=c_+.
\end{array}\right\}$$
Since the number of boundary points of a compact one dimensional
manifold is even, we conclude that the modulo two number of broken
flow lines $[w^1]\#[w^2]$ subject to the asymptotic condition
(\ref{as1}) coincides with the modulo two number of unparametrized
nonconstant broken flow lines satisfying
\begin{equation}\label{as2}
\left.\begin{array}{c}
\mathrm{ev}_-([w^1])=\ell(c_-),\\
\mathrm{ev}_+([w^1])=\mathrm{ev}_-([w^2]),\\
\mathfrak{w}\big(\mathrm{ev}_+([w^1])\big) \neq 0,\\
\mathfrak{w}\big(\mathrm{ev}_+[w^2])\big)=0, \\
P\mathrm{ev}_+[w^2]=c_+.
\end{array}\right\}
\end{equation}
In order to ease notation we abbreviate
$$\gamma=\mathrm{ev}_+([w^1]) \in
\mathrm{crit}(\mathcal{A}^\mu_H).$$ We first note that there exists
$$\varepsilon \in \{0,1\}$$
such that
$$\mu_{CZ}(\ell(c_-))-\mu_{CZ}(\gamma)=1+\varepsilon, \quad
\mu_{CZ}(\gamma)-\mu_{CZ}(\ell(c_+))=1-\varepsilon.$$ The reason for
the ambiguity in the index computation lies in the fact that
$\gamma$ is not an isolated critical point of $\mathcal{A}_H^\mu$
but lies in a circle family of critical points. If $P \colon
\widetilde{\mathscr{L}}_E \times \mathbb{R}\to
\widetilde{\mathscr{L}}_M$ is the projection we have
$$\mu_{CZ}(\gamma)=\mu_{CZ}(P\gamma)-2\mathfrak{w}(\gamma).$$
Since the winding numbers of $\ell(c_-)$ and $\ell(c_+)$ vanish, we
get
$$\mu_{CZ}(\ell(c_-))=\mu_{CZ}(c_-), \quad \mu_{CZ}(\ell(c_+))=
\mu_{CZ}(c_+).$$ Using these facts we compute
$$\mu_{CZ}(c_-)-\mu_{CZ}(P\gamma)=\mu_{CZ}(\ell(c_-))-\mu_{CZ}(\gamma)
-2\mathfrak{w}(\gamma)=1+\varepsilon-2\mathfrak{w}(\gamma)$$ and
$$\mu_{CZ}(P\gamma)-\mu_{CZ}(c_+)=\mu_{CZ}(\gamma)+2\mathfrak{w}(\gamma)
-\mu_{CZ}(\ell(c_+))=1-\varepsilon+2\mathfrak{w}(\gamma).$$ We now
consider the broken flow line $[Pw^1]\#[Pw^2]$ of $\nabla_J
\mathcal{A}_H$. Since $J \in \mathcal{J}^\nu_{\mathrm{reg}}$ there
are three cases to distinguish.
\\ \\
\textbf{Case~1: }The two flow lines $Pw^1$ and $Pw^2$ are not
constant. In this case
$$\mu_{CZ}(c_-)-\mu_{CZ}(P\gamma)=1, \quad
\mu_{CZ}(P\gamma)-\mu_{CZ}(c_+)=1.$$ Hence
$$\varepsilon=0, \quad \mathfrak{w}(\gamma)=0.$$
Therefore the broken flow line $[w^1]\#[w^2]$ does not satisfy the
asymptotic condition (\ref{as2}).
\\ \\
\textbf{Case~2: } The flow line $Pw^1$ is constant. In this case
$P\gamma=c_-$ and
$$\mu_{CZ}(c_-)-\mu_{CZ}(P\gamma)=0, \quad
\mu_{CZ}(P\gamma)-\mu_{CZ}(c_+)=2.$$ Hence
$$\varepsilon=1, \quad \mathfrak{w}(\gamma)=1.$$
Therefore the broken flow line $[w^1]\#[w^2]$ satisfies the
asymptotic condition (\ref{as2}).
\\ \\
\textbf{Case~3: }The flow line $Pw^2$ is constant. In this case
$P\gamma=c_+$ and
$$\mu_{CZ}(c_-)-\mu_{CZ}(P\gamma)=2, \quad
\mu_{CZ}(P\gamma)-\mu_{CZ}(c_+)=0.$$ Hence
$$\varepsilon=1, \quad \mathfrak{w}(\gamma)=0.$$
Since $P\gamma=c_+$ and $\mathfrak{w}(\gamma)=0$ we conclude that
$\gamma \in S^1 \ell(c_+)$. Since the action of a nonconstant
gradient flow line is strictly decreasing we conclude that $w^2$
itself is constant, which contradicts our assumption. Hence Case~3
never occurs.
\\ \\Summarizing we have shown that the only broken flow lines
$[w^1]\#[w^2]$ meeting the asymptotic condition (\ref{as2}) are the
one's from Case~2, i.e.~$Pw^1$ is constant and
$\mathfrak{w}(\gamma)=1$. We conclude that $w^1$ is a vortex and
$w^2$ is an Oni. Since the vortex number is one by
Theorem~\ref{vonu} we deduce that the modulo two number of broken
gradient flow lines subject to condition (\ref{as2}) coincides with
the number of Onis. This finishes the proof of the Theorem. \hfill
$\square$

\subsection{Proof of Proposition~\ref{generic}}

To establish that for generic $B \in \mathfrak{B}(J)$ regularity
conditions (i) and (ii) are met, we first observe that since $J \in
\mathcal{J}^\nu_{\mathrm{reg}}$ the linearization of the gradient
flow equation (\ref{gr2}) along a finite energy gradient flow line
$w$ is already surjective in the horizontal directions. To show that
it is also surjective in the vertical directions we distinguish two
cases. In the first $Pw$ is nonconstant, i.e. $Pw$ is a finite
energy solution of Floer's gradient flow equation. We claim that $Pw$ necessarily leaves the neighborhood $\mathcal{U}$. We recall that $\mathcal{U}$ is the union of disjoint neighborhoods of the periodic orbits of $X_H$ where each such neighborhood contracts onto the periodic orbit, see the discussion before equation \eqref{eqn:space_of_Bs}. If $Pw$ is contained in $\mathcal{U}$ then it has to be a gradient trajectory connecting the same periodic orbit with cappings $d$ and $d\#Pw$. Since $Pw$ is contained in $\mathcal{U}$ the two cappings are homotopic to each other and thus $Pw$ is constant.

By \cite[Theorem
4.3]{floer-hofer-salamon} the set of regular points for $Pw$ is open
and dense. In view of Lemma~\ref{transv} below a standard argument,
see for instance \cite[Section 5]{floer-hofer-salamon} or
\cite[Chapter 3]{mcduff-salamon1} establishes that for generic $B
\in \mathfrak{B}(J)$ the linearization of the gradient flow equation
is also vertically surjective. In the second case $Pw$ is constant.
But then $w$ is a vortex and vortices are by
Proposition~\ref{transverse} always transverse independent of the
perturbation $B \in \mathfrak{B}(J)$. This shows that generically
(i) and (ii) hold true.

The rest of the proof is devoted to show that generically the
regularity condition (iii), i.e. the vanishing of
$J_t^B$-holomorphic spheres, is satisfied as well. Assume that $u
\colon S^2 \to E$ is a $J_t^B$-holomorphic sphere for some $t \in
S^1$. Its shadow $v=p \circ u \colon S^2 \to M$ is a
$J_t$-holomorphic sphere. Hence if $\beta \in \Omega^2(M)$ is a
closed two-form representing the first Chern class of $TM$ we obtain
in local holomorphic coordinates of $S^2$ the inequality
$$|\beta(\partial_x v,\partial_y v)|=|\beta(\partial_x v,J_t(v)
\partial_x v)| \leq \kappa_\beta(J_t)||\partial_x v||^2_{J_t}
=\kappa_\beta(J_t)\omega(\partial_x v, \partial_y v)$$ and therefore
after integration
$$|\langle c_1(v^*TM),[S^2]\rangle| \leq \kappa_\beta(J_t)\omega([v]).$$
Since $\beta$ was an arbitrary representative of the first Chern
class we get by definition of the Auroux constant
$$|\langle c_1(v^*TM),[S^2]\rangle| \leq \kappa(J_t)\omega([v]).$$
Therefore we conclude
\begin{eqnarray}\label{auroux}
\langle c_1(u^*TE),[S^2]\rangle&=&\langle c_1(v^*TE),[S^2]\rangle\\
\nonumber &=&\langle c_1(v^*TM),[S^2]\rangle+\langle
c_1(v^*E),[S^2]\rangle\\ \nonumber
&\leq&(\kappa(J_t)-\nu)\omega([v])\\ \nonumber
&<&\min\big\{2-n,0\big\}\omega([v])
\end{eqnarray}
where for the last inequality we used that
$\nu>\max\{n+\kappa(J_t)-2,\kappa(J_t)\}$ since $J \in
\mathcal{J}^\nu$.
 The virtual dimension
$\mathrm{virdim}_u\big(\mathfrak{H}(J^B_t)\big)$ of the moduli space
$\mathfrak{H}(J^B_t)$ of $J^B_t$-holomorphic spheres at $u$ is given
by the Riemann Roch formula
$$
\mathrm{virdim}_u\big(\mathfrak{H}(J^B_t)\big)=
2\mathrm{dim}(E)+2\langle c_1(u^*TE),[S^2]\rangle-6$$ and hence can
be estimated using (\ref{auroux})
\begin{equation}\label{virdim1}
\mathrm{virdim}_u\big(\mathfrak{H}(J^B_t)\big)<
2n-4-2\max\big\{n-2,0\big\} \omega([v]).
\end{equation}
Since $u$ is a nonconstant $J^B_t$-holomorphic curve it holds that
$$\omega([v])=\omega_E([u])>0.$$
Since $\omega$ is integral, it follows that
\begin{equation}\label{integral}
\omega([v]) \geq 1.
\end{equation}
Inequalities (\ref{virdim1}) and (\ref{integral}) imply that
\begin{equation}\label{vird}
\mathrm{virdim}_u\big(\mathfrak{H}(J^B_t)\big)<
2n-4-2\max\big\{n-2,0\big\} \leq 0.
\end{equation}
Since the virtual dimension is an even integer we immediately obtain
from (\ref{vird}) the stronger estimate
\begin{equation}\label{virdim2}
\mathrm{virdim}_u\big(\mathfrak{H}(J_t^B)\big) \leq -2.
\end{equation}
Now assume that $u$ is simple in the sense of \cite[Chapter
2.5]{mcduff-salamon1}. It follows from Lemma~\ref{simple} below,
that $u$ is horizontally injective on a dense set, in particular, there are horizontally injective points on $E\setminus{p^{-1}(\mathcal{U})}$. Therefore the usual
transversality arguments as explained in \cite[Chapter
3.2]{mcduff-salamon1} imply that for generic choice of the
perturbation $B$ the moduli space of simple $J_t^B$ is a manifold
whose dimension equals its virtual dimension. Therefore by
(\ref{virdim2}) generically there are no simple $J^B_t$-holomorphic
curves. Hence there are no nonconstant $J^B_t$-holomorphic curves at
all, since every nonconstant $J^B_t$-holomorphic curve has an
underlying simple curve. This finishes the proof of the Proposition.
\hfill $\square$
\\ \\
It remains to show two lemmas which were used in the proof above.

\begin{lemma}\label{transv}
Assume $e \in E$, $h_0 \in \HH_e$, $v_0 \in \VV_e$ such that $h_0 \neq
0$,  and $t_0 \in S^1$. Then there exists $B \in \mathfrak{B}(J)$
such that $B_{t_0} h_0=v_0$.
\end{lemma}
\textbf{Proof: } By local triviality there exists a neighborhood $U$
of $e$ and sections $h \in \Gamma(U,\HH|_{U})$ and $v \in
\Gamma(U,\VV|_{U})$ with the property that $h$ is not vanishing and
$$h(e)=h_0, \quad v(e)=v_0.$$
Choose further a compactly supported function $\beta \in
C^\infty(U,[0,1])$ satisfying $\beta(e)=1$. Since $J_t$ is
$\omega$-compatible the orthogonal complement $\langle h, Jh
\rangle^{\perp_t} \subset \HH$ with respect to the metric $\omega(
\cdot, J_t \cdot)$ is invariant under $J_t$. We define $B \in
\Gamma_0\big(E,\mathrm{End}(\HH,\VV)\big)$ as the section which vanishes
outside $U$ and on $U$ is determined by
$$B_th=\beta v, \quad B_tJ_th=-\beta Iv, \quad B|_{\langle h, J_th \rangle
^{\perp_t}}=0.$$ By construction we have $B_tJ_t=-IB_t$ and
$B_{t_0}h_0=v_0$. This finishes the proof of the Lemma. \hfill
$\square$
\\ \\
To state the second lemma we first need a definition. For $u \in
C^\infty(S^2,E)$ we denote for $z \in S^2$ by
$$d^h u(z) \colon T_z S^2 \to \\\HH_{u(z)}$$
the composition of $du(z)$ with the projection from $T_{u(z)}E$ to
$\HH_{u(z)}$ along $\VV_{u(z)}$.
\begin{fed}
A $J^B$-holomorphic sphere $u \colon S^2 \to E$ is called
\emph{somewhere horizontally injective} if there exists $z \in S^2$
such that
$$d^hu(z) \neq 0, \quad u^{-1}(u(z))=\{z\}.$$
\end{fed}
\begin{lemma}\label{simple}
Assume that $u \colon S^2 \to E$ is a simple $J^B$-holomorphic
curve. Then $u$ is horizontally injective on a dense set.
\end{lemma}
\textbf{Proof: }Denote by $I \subset S^2$ the subset of injective
points of $S^2$, by $S \subset S^2$ the subset of horizontally
injective points and by $R(p(u)) \subset S^2$ the subset of
nonsingular points of $p(u) \colon S^2 \to M$. Then
\begin{equation}\label{hi}
S=I \cap R(p(u)).
\end{equation}
We first observe that $p(u)$ is not constant, since otherwise $u$
would lie in one fibre and hence itself must be constant,
contradicting the assumption that it is simple. Hence it follows
from \cite[Lemma 2.4.1]{mcduff-salamon1} that the complement of
$R(p(u))$ is finite. Moreover, it follows from \cite[Proposition
2.5.1]{mcduff-salamon1} that the complement of $I$ is countable.
Hence by (\ref{hi}) the complement of $S$ is countable. In
particular, $S$ is dense. \hfill $\square$

\subsection{Proof of Theorem~\ref{finite}}

If $w$ is a flow line of $\nabla_B \mathcal{A}^\mu_H$, then $Pw$ is
a flow line of $\nabla_J \mathcal{A}^H$. Hence both functional
$\mathcal{A}^\mu_H$ and $\mathcal{A}_H \circ P$ are decreasing along
$w$. Therefore using that the perturbation is regular
Theorem~\ref{finite}  follows from the usual breaking arguments in
Morse homology, see \cite{schwarz}, as soon as the following Theorem
is established.

\begin{thm}\label{compact}
Assume that $J \in \mathcal{J}^\nu$ and $B \in
\mathfrak{B}_{\mathrm{reg}}^T(J)$. Suppose further that
$w_\nu=([u_\nu,\bar{u}_\nu],\eta_\nu)$ for $\nu \in \mathbb{N}$ is a
sequence of flow lines of $\nabla_B \mathcal{A}^\mu_H$ for which
there exists $a<b$ with the property that
$$a \leq \mathcal{A}^\mu_H(w_\nu)(s) \leq b,\quad
a \leq \mathcal{A}_H(P w_\nu)(s) \leq b, \quad \nu \in
\mathbb{N},\,\,s \in \mathbb{R}.$$ Then there exists a subsequence
$\nu_j$ and a flow line $w$ of $\nabla_B \mathcal{A}^\mu_H$ such
that $w_{\nu_j}$ converges in the $C^\infty_{\mathrm{loc}}$-topology
to $w$.
\end{thm}
\textbf{Proof: } The proof of Theorem~\ref{compact} follows along
standard lines, see \cite[Chapter 4]{mcduff-salamon1} if the
following three conditions for $w_\nu$ can be established.
\begin{description}
 \item[(i)] A uniform $C^0$-bound on the Lagrange multipliers $\eta_\nu$.
 \item[(ii)] A uniform $C^0$-bound for the loops $u_\nu$.
 \item[(iii)] A uniform bound on the derivatives of the loops
 $u_\nu$.
\end{description}
Condition (i) is the content of Proposition~\ref{lagbound},
condition (ii) is the content of Proposition~\ref{loopbound}, and
condition (iii) follows because there is no bubbling, since the
perturbation $B$ is regular. This proves the theorem. \hfill
$\square$

\begin{prop}\label{lagbound}
Under the assumptions of Theorem~\ref{compact} there exists a
constant $c=c(a,b)$ such that $|\eta_\nu(s)| \leq c$ for every $\nu
\in \mathbb{N}$ and every $s \in \mathbb{R}$.
\end{prop}
\textbf{Proof: } If $w=([u,\bar{u}],\eta)$ is a critical point of
$\mathcal{A}^\mu_H$, then a computation shows that
$$\eta=\mathcal{A}^\mu_H(w)-\mathcal{A}_H(Pw).$$
By assumption $\mathcal{A}^\mu_H-\mathcal{A}_H \circ P$ is uniformly
bounded along the gradient flow lines $w_\nu$. By applying the
arguments from \cite{cieliebak-frauenfelder} to $\mathcal{A}^\mu_H-
\mathcal{A}_H \circ P$ instead of $\mathcal{A}^\mu_H$ the Lagrange
multipliers $\eta_\nu$ can be bounded in terms of
$\mathcal{A}^\mu_H-\mathcal{A}_H \circ P$ which gives a uniform
bound for them. A similar argument was used in
\cite{cieliebak-frauenfelder-paternain}. For complete details we
refer to \cite{frauenfelder1}. \hfill $\square$

\begin{prop}\label{loopbound}
Under the assumptions of Theorem~\ref{compact} there exists a
compact subset $K=K(a,b) \subset E$ such that $u_\nu(s,t) \in K$ for
every $\nu \in \mathbb{N}$ and every $(s,t) \in \mathbb{R}\times
S^1$.
\end{prop}
\textbf{Proof: }Since the perturbation $B$ is compactly supported
$(E,J^B_t)$ is convex at infinity. A Laplace estimate for the
gradient flow equation of Rabinowitz action functional then implies
that $u_\nu$ stay in a bounded set of $E$. This Laplace estimate was
also used in \cite{cieliebak-frauenfelder-oancea}. We refer to
\cite{cieliebak-frauenfelder-oancea} or \cite{frauenfelder1} for
complete details. \hfill $\square$

\section{Homotopies of homotopies}

In the following we assume that
$\nu>\max\{n+\kappa(\omega)-1,\kappa(\omega)\}$ which allows us to
conclude that for generic homotopies of homotopies there are still
no holomorphic spheres. Instead of studying a fixed Hamiltonian $H$,
we consider now the continuation from a $C^2$-small Morse function
$H_0$ to the Hamiltonian $H$. We can use the gradient flow equation
of Rabinowitz action functional to define continuation homomorphisms
$$\Phi \colon CF_*(H_0) \to CF_*(H), \quad \Psi \colon
CF_*(H) \to CF_*(H_0).$$ Namely choose a smooth monotone cutoff
function $\beta \in C^\infty(\mathbb{R},[0,1])$ satisfying
$$\beta(s)=0, \,\, s<-1, \qquad \beta(s)=1,\,\,s>1$$
and consider the $s$-dependent Hamiltonians
$$H^+_s=\beta(s)H+(1-\beta(s))H_0, \quad
H^-_s=\beta(s)H_0+(1-\beta(s))H$$ Then $\Phi$ is defined by counting
gradient flow lines of the $s$-dependent Rabinowitz action
functional $\mathcal{A}^\mu_{H^+}$ between a lift of a critical
point of $\mathcal{A}_{H_0}$ and a lift of a critical point of
$\mathcal{A}_H$ and similarly for $\Psi$ one counts gradient flow
lines of $\mathcal{A}^\mu_{H^-}$. To study their composition $\Psi
\circ \Phi$ we consider as usual in Floer homology a homotopy of
homotopies. Namely for $r \in [0,1]$ choose a smooth family of
functions $\beta_r \in C^\infty(\mathbb{R},[0,1])$ which satisfy the
following conditions
\begin{itemize}
 \item $\beta_0=0$, $\beta_1=1$,
 \item $\beta_r$ is compactly supported for $r<1$,
 \item For $r \in [0,1]$ the functions $\beta_r$ are monotone
 increasing for $s<0$ and monotone decreasing for $s>0$,
 \item The time-shifted functions $\big(\frac{1}{r-1}\big)_*\beta_r$
 defined by $\big(\frac{1}{r-1}\big)_*\beta_r(s)
 =\beta_r\big(s+\frac{1}{r-1}\big)$ for $s \in \mathbb{R}$
 converge in the
 $C^\infty_{\mathrm{loc}}$-topology to $\beta$ as $r$ goes to $1$.
 \item The time-shifted functions $\big(\frac{1}{1-r}\big)_*\beta_r$
 converge in the
 $C^\infty_{\mathrm{loc}}$-topology to $1-\beta$ as $r$ goes to $1$.
\end{itemize}
For $r \in [0,1)$ we consider the family of $s$-dependent
Hamiltonians
$$H_r=H_0+\beta_r (H-H_0).$$
 A \emph{Homotopy-Oni} is a pair $(r,w)$ where $r \in
[0,1)$ and $w$ is a finite energy gradient flow line of the time
dependent Rabinowitz action function $\mathcal{A}^\mu_{H_r}$ whose
asymptotic winding numbers satisfy
$\mathfrak{w}(\mathrm{ev}_-(w))=1$ and
$\mathfrak{w}(\mathrm{ev}_+(w))=0$. Counting Homotopy-Onis gives
rise to a linear map
$$\mathfrak{O} \colon CF_*(H_0) \to CF_*(H_0).$$
We denote by
$$\partial \colon CF_*(H_0) \to CF_{*-1}(H_0)$$
the boundary operator obtained by counting Morse gradient flow lines
of $H_0$ on $M$ with respect to a Morse-Smale metric on $M$ coming
from a regular $\omega$-compatible almost complex structure. We
refer to \cite{le-ono,salamon-zehnder} for an existence proof of
such metrics.
\begin{conj}\label{homhom}
There exist linear maps $T \colon CF_*(H_0) \to CF_{*+1}(H_0)$ such
that
$$\Psi \circ \Phi=\mathrm{id}|_{CF_*(H_0)}+T \partial+\partial T
+\mathfrak{O}.$$
\end{conj}
The maps $T$ in the conjecture arise in the same way as in the usual
homotopy of homotopies argument in Floer homology, see \cite[Section
3.4]{salamon}. The proof of the conjecture should follow by
basically the same argument as in the proof of Theorem~\ref{oni}, up
to one point which involves abstract perturbation theory. This
concerns the identification of the map $\mathfrak{R}(H_0) \colon
CF_*(H_0) \to CF_*(H_0)$ obtained by counting gradient flow lines of
Rabinowitz action functional $\mathcal{A}^\mu_{H_0}$ with the
boundary operator in Morse homology. If the symplectic manifold is
not semipositive, then it is hard to image that such a result can be
proved without the help of abstract perturbation theory. In ordinary
Floer homology this was proved for example in \cite[Section
22]{fukaya-ono}. We expect that the argument of Fukaya and Ono can
be adjusted to our situation which then leads to a proof of the
conjecture.

We like to point out that the Fukaya-Ono argument using multisections proves that the Floer differential for the $C^2$-small $H_0$ agrees with the Morse differential. The Morse differential can be counted with integer coefficients. In their further constructions $\mathbb{Q}$-coefficients are essential whereas our approach works over $\mathbb{Z}/2$-coefficients.

To deduce some useful information from Conjecture~\ref{homhom} one
needs to have some information of the Oni operator. Up to now we
only have some clue on it under the additional assumption that the
Hamiltonian satisfies the condition
\begin{equation}\label{sqr}
H_{t+\frac{1}{2}}=H_t,
\end{equation}
i.e. the time one map $\phi_H$ of the Hamiltonian flow admits a
square root.
\begin{conj}\label{h3}
Assume that the Hamiltonian $H$ satisfies (\ref{sqr}), and that
$\nu$ is even and bigger than $n+\kappa(\omega)$. Then there exist
linear maps
$$F \colon CF_*(H_0) \to
\bigoplus_{i=-2}^1 CF_{*+i}(H), \quad G \colon
\bigoplus_{i=-2}^1CF_{*+i}(H)\to CF_*(H_0),$$ and $$S \colon
CF_*(H_0) \to CF_{*-1}(H_0)$$ such that
\begin{equation}\label{h3eq}
\mathfrak{O}=G \circ F+\partial S+S \partial.
\end{equation}
\end{conj}
We give an outline of the proof of Conjecture~\ref{h3}. To consider
a homotopy of homotopies we need a family of almost complex
structures indexed by the homotopy parameter $r \in [0,1]$ which in
addition might depend on the parameters $s$ and $t$. If this family
satisfies
\begin{equation}\label{sqrJ}
J^B_{t+\frac{1}{2},s,r}=J^B_{t,s,r}
\end{equation}
we get an involution on the Homotopy-Oni given by rotating the loop
on $E$ by 180 degrees. Since $\nu$ is even this involution keeps
critical points of $\mathcal{A}^\mu_{H_0}$ of winding number zero
fixed but acts freely on critical points of winding number one.
Therefore this involution is free on the Homotopy-Onis. On the other
in general there is little hope to achieve transversality by keeping
condition (\ref{sqrJ}). To overcome this difficulty we proceed a bit
different. Choose a lift $\ell \colon
\mathrm{crit}(\mathcal{A}_{H_0}) \to
\mathrm{crit}(\mathcal{A}^\mu_{H_0})$ which is a section for the
projection $P \colon \mathrm{crit}(\mathcal{A}_{H_0}^\mu) \to
\mathrm{crit}(\mathcal{A}_{H_0})$ and satisfies
$\mathfrak{w}(\ell(c))=1$ for every critical point $c \in
\mathrm{crit}(\mathcal{A}_{H_0})$. In the following we drop the
subscripts indicating the dependence of the families of almost
complex structures on the $s$ and $r$-parameters. If $\theta \in
S^1$ and $J^B$ is a family of almost complex structures not
necessarily satisfying (\ref{sqrJ}) we set
$$(\theta_*J^B)_t=J^B_{t+\theta}.$$
Recall that on $\mathscr{L}_E$ the circle acts by rotating the loop.
\begin{fed}
A \emph{Married Homotopy-Oni} is a tuple $(w,r)$ where $r \in [0,1)$
and $w$ is a finite energy gradient flow line of $\nabla_{\theta_*
B}\mathcal{A}^\mu_{H_r}$ for some $\theta \in S^1$ satisfying
$\mathrm{ev}_-(w) \in \theta_*
\ell(\mathrm{crit}(\mathcal{A}_{H_0}))$ and
$\mathfrak{w}(\mathrm{ev}_+(w))=0$.
\end{fed}
Married Homotopy-Onis are harmless since each Married Homotopy-Oni
has a partner obtained by rotating the loops by 180 degrees.
Therefore the Oni operator $\mathfrak{O}^m$ obtained by counting
Married Homotopy-Onis modulo two vanishes
$$\mathfrak{O}^m=0\colon CF_*(H_0) \to CF_*(H_0).$$
We next construct a homotopy between Homotopy-Onis and Married
Homotopy-Onis. To this end we first choose a smooth homotopy
$J^B_{\theta,\rho}$ where $\theta \in S^1$ and $\rho \in [0,1]$ such
that
$$J^B_{\theta,0}=J^B, \quad J^B_{\theta,1}=\theta_*J^B.$$
Note that since $J^B$ already depends on the three parameters $t$,
$s$ and $r$ we have no a five-parameter family of almost complex
structures on $E$. We apologize for any inconvenience this might
cause to the reader.
 Since $\nu>\kappa(\omega)+n$ for generic
choice of this five parameter family of almost complex structures
there are no holomorphic spheres. In the following we abbreviate for
$(\theta,\rho) \in S^1 \times [0,1]$
$$\nabla_{\theta,\rho}=\nabla_{B_{\theta,\rho}}.$$Note that the critical manifold of
$\mathcal{A}^\mu_H$ consists of a disjoint union of circles. We
choose a Morse function
$$h \colon \mathrm{crit}(\mathcal{A}^\mu_H) \to \mathbb{R}$$
with the property that $h$ restricted to each circle has precisely
one maximum and one minimum. We further choose a Riemannian metric
on $\mathrm{crit}(\mathcal{A}^\mu_H)$ and denote by
$$\phi^\tau_{\nabla h} \colon \mathrm{crit}(\mathcal{A}^\mu_H) \to
\mathrm{crit}(\mathcal{A}^\mu_H), \quad \tau \in \mathbb{R}$$ the
gradient flow of $h$ on the critical manifold of
$\mathcal{A}^\mu_H$.
\begin{fed}
A \emph{Kibidango}\footnote{Eating Kibidangos Momotaro was able to fight the onis.} is an $m$-tuple
$$\mathfrak{k}=(k^i)_{1 \leq i \leq m}$$
for some positive integer $m$ satisfying the following properties.
\begin{description}
 \item[(i)] If $m=1$, then $k^1=(w,r,\rho)$ is a triple, where $r \in [0,1)$,
 $\rho \in [0,1]$, and $w$ is a finite energy gradient flow line of
 $\nabla_{\theta,\rho}\mathcal{A}^\mu_{H_r}$ with positive asymptotic
 satisfying $\mathfrak{w}(\mathrm{ev}_+(w))=0$ and $\theta \in S^1$
 determined by $\mathrm{ev}_-(w) \in
 \theta_*\ell(\mathrm{crit}\mathcal{A}_{H_0})$.
 \item[(ii)] If $m>1$, then $k^i=(z^i,\rho^i)$ is a tuple
 for each $i \in \{1,\cdots,m\}$, with the following properties
 \begin{description}
 \item[(a)] $0 \leq \rho^1 \leq \rho^2 \leq \cdots \leq \rho^m \leq
 1$,
 \item[(b)] $z^1=w^1$ is a finite energy flow line of
 $\nabla_{\theta,\rho^1}\mathcal{A}^\mu_{H^+}$ where
 $\theta=\theta(\mathfrak{k})$ is determined by $\mathrm{ev}_-(w^1)
 \in \theta_*\ell(\mathrm{crit}(\mathcal{A}_{H_0})$,
 \item[(c)] $z^i=[w^i]$ is an unparametrised finite energy flow line of
 $\nabla_{\theta,\rho^i}\mathcal{A}^\mu_H$ for $1<i<m$,
 \item[(d)] $z^m=w^m$ is a finite energy flow line of
 $\nabla_{\theta,\rho^m}\mathcal{A}^\mu_{H^-}$ whose positive asymptotic
 satisfies $\mathfrak{w}(\mathrm{ev}_-(w^m))=0$,
 \item[(e)] if $0<\rho_i\leq \rho_{i+1}<1$ for $1 \leq i <m$, then
 $\mathrm{ev}_+(w_i)=\mathrm{ev}_-(w_{i+1})$,
 \item[(f)] if $\rho_i=1$ for $1 \leq i <m$, then $\mathrm{ev}_+(w^i)
 \notin \mathrm{crit}(h)$ and there exists $\tau \geq 0$ such that
 $\phi^\tau_{\nabla h}(\mathrm{ev}_+(w^i))=\mathrm{ev}_-(w^{i+1})$,
 \item[(g)] if $\rho_i=0$ for $1<i \leq m$, then $\mathrm{ev}_-(w^i)
 \notin \mathrm{crit}(h)$ and there exists $\tau \geq 0$ such that
 $\phi^\tau_{\nabla h}(\mathrm{ev}_+(w^{i-1}))=\mathrm{ev}_-(w^i)$.
 \end{description}
\end{description}
\end{fed}
Kibidangos interpolate between Homotopy-Onis and Married
Homotopy-Onis. However, the moduli space of Kibidangos does not need
to be compact. Namely Kibidangos might break at critical points of
$\mathcal{A}^\mu_{H_0}$ or critical points of the Morse function $h$
on the critical manifold of $\mathcal{A}^\mu_H$. The first occurence
should give rise to the term $\partial S +S \partial$ in
(\ref{h3eq}). However, to make this precise one has to relate
gradient flow lines of $\mathcal{A}^\mu_{H_0}$ with Morse gradient
flow lines of $H_0$ which requires abstract perturbation theory and
a generalization of the Theorem of Fukaya and Ono to our set-up. To
see at which critical points of $h$ a Kibidango can break one has to
analyze again the shadow of a Kibidango under the projection $P
\colon \widetilde{\mathscr{L}}_E \times \mathbb{R} \to
\widetilde{\mathscr{L}}_M$. By looking at the indices it turns out
that generically breaking can only happen at winding number 0 or 1.
Hence for each critical point of $\mathcal{A}^\mu_H$ there are four
critical points of $h$ at which breaking might occur, namely the
maximum and minimum of $h$ on the two circles corresponding to
winding number 0 and winding number 1. Hence we can identify these
points with vectors in $\bigoplus_{i=-2}^1 CF_{*+1}(H)$ and the
broken Kibidangos give rise to the maps $F$ and $G$ in (\ref{h3eq}).
This finishes the outline of the proof of Conjecture~\ref{h3}.
\hfill $\square$
\\ \\
As a consequence of Conjecture~\ref{homhom} and
Conjecture~\ref{h3} we obtain the following Corollary.
\begin{cor}[assuming Conjectures \ref{homhom} and \ref{h3}]
Assume that $\phi$ is a nondegenerate Hamiltonian symplectomorphism
which has a square root. Then
$$\# \mathrm{Fix}(\phi) \geq
\frac{1}{5}\sum_{k=0}^{2n}b_k(M;\mathbb{Z}_2)$$ where
$b_k(M;\mathbb{Z}_2)$ are the $\mathbb{Z}_2$-Betti numbers of $M$ and $\mathrm{Fix}(\phi)$ is the set of contractible fixed points.
\end{cor}

\begin{proof}
 The proof follows from next Proposition as follows. So far, we worked with the Novikov ring over
 $$
 \Gamma=\frac{\pi_2(M)}{\ker\omega\cap\ker c_1}\;.
 $$
 This as the advantage of having a well-defined grading. Instead, if one uses the Novikov ring over
 $$
 \Gamma_0=\frac{\pi_2(M)}{\ker\omega}
 $$
 then Floer homology looses its $\mathbb{Z}$-grading. On the other hand the Novikov ring over $\Gamma_0$ is a field $\Lambda$. The corresponding chain groups are denoted by $CF(H)$. Thus,
 $$
 \#\mathrm{Fix}(\phi)=\dim_\Lambda CF(H).
 $$
 Moreover, $\dim_\Lambda HF(H_0)=\sum_{k=0}^{2n}b_k(M;\mathbb{Z}_2)$. We set
 $$
 V:=CF(H_0),\;W:=CF(H),\; X:=W\oplus W\oplus W\oplus W\;.
 $$
 Furthermore, we set $R:=T+S$. Then we conclude from the next Proposition that
 $$
 \sum_{k=0}^{2n}b_k(M;\mathbb{Z}_2)=\dim H(V,\partial)\leq 5\dim W=5\# \mathrm{Fix}(\phi)\;.
 $$ 
\end{proof}

\begin{prop}
Let $V$, $W$, and $X$ finite dimensional vector spaces over some fixed field. We consider maps 
$$
\begin{aligned}
\Phi:V&\to W\\
\Psi:W&\to V\\
F:V&\to X\\
G:X&\to V\\
\partial:V&\to V\\
R:V&\to V\\
\end{aligned}
$$
satisfying $\partial^2=0$ and 
$$
\Psi\Phi+GF=\mathrm{id}_V+R\partial+\partial R\;.
$$
Then the following inequality holds
$$
\dim H(V,\partial)\leq\dim W+\dim X\;.
$$
\end{prop}

\begin{proof}
First we explain that we may assume that $X=\{0\}$ is the trivial vector space. For that we set
$$
\begin{aligned}
\widetilde\Phi:V&\to W\oplus X \\
v&\mapsto (\Phi v, Fv)\\[1ex]
\widetilde\Psi:W\oplus X&\to V \\
(w,x)&\mapsto \Psi w+Gx
\end{aligned}
$$
and compute
$$
\widetilde\Psi\widetilde\Phi=\Psi\Phi+GF=\mathrm{id}_V+R\partial+\partial R\;.
$$
Thus, the assertion of the Proposition in the case $X=\{0\}$ implies the general case since $\dim W\oplus X=\dim W+\dim X$. It remains to prove
$$
\dim H(V,\partial)\leq\dim W\;.
$$
whenever
$$
\Psi\Phi=\mathrm{id}_V+R\partial+\partial R\;.
$$
First we show that
$$
\ker\big(\Phi|_{\ker\partial}\big)\subset\mathrm{im}\: \partial\;.
$$
Indeed, if $v\in V$ satisfies $\partial v=0$ and $\Phi v=0$ then
$$
0=\Psi\Phi v=v+R\partial v+\partial Rv=v+\partial Rv\;.
$$
Then we can estimate
$$
\begin{aligned}
\dim H(V,\partial)&=\dim\ker\partial-\dim \mathrm{im}\: \partial \\
&=\underbrace{\dim\ker\big(\Phi|_{\ker\partial}\big)}_{\leq\dim \mathrm{im}\:\partial}+\dim\mathrm{im}\:\big(\Phi|_{\ker\partial}\big)-\dim \mathrm{im}\: \partial \\
&\leq \dim\mathrm{im}\:\big(\Phi|_{\ker\partial}\big)\\
&\leq \dim W
\end{aligned}
$$
and this proves the Proposition.
\end{proof}


\appendix

\section{Vortices}

\subsection{The vortex equation}

As vortices we refer to solutions $(u,\eta) \in
C^\infty(\mathbb{R}\times S^1,\mathbb{C}) \times
C^\infty(\mathbb{R},\mathbb{R})$ of the problem
\begin{equation}\label{vortex}
\left.\begin{array}{c}
\partial_s u+i\partial_t u-2\pi \eta u=0 \\
\partial_s \eta-\pi\int_0^1|u|^2(t,\cdot)dt+\pi=0\\
\end{array}\right\}
\end{equation}
whose energy is finite
$$E(u,\eta):=\int_{\mathbb{R}\times S^1} |\partial_s u|^2 ds dt+
\int_{\mathbb{R}}|\partial_s \eta|^2 ds<\infty.$$ Vortices arise as
gradient flow lines of Rabinowitz action functional
$$\mathcal{A}^\mu \colon C^\infty(S^1,\mathbb{C}) \times
\mathbb{R}\to \mathbb{R}$$ given by
$$\mathcal{A}^\mu(u,\eta)=-\int u^*\lambda-\eta \int
\mu(u) dt$$ where $\lambda=x dy$ is the primitive of the standard
symplectic structure on $\mathbb{C}$. The gradient is taken with
respect to the product metric on $C^\infty(S^1,\mathbb{C}) \times
\mathbb{R}$ which on the first factor is given by the $L^2$-metric
and on the second by the standard inner product of $\mathbb{R}$.

The differential of Rabinowitz action function $d\mathcal{A}^\mu$ is
invariant under the $S^1 \times \mathbb{Z}$-action on
$C^\infty(S^1,\mathbb{C}) \times \mathbb{R}$ given by
$$(r,k)_*(v,\eta)=((r,k)_*v, \eta+k), \quad (r,k) \in S^1 \times \mathbb{Z}$$
where
$$(r,k)_*v(t)=e^{-2\pi i r} e^{-2\pi i k t}v(t),\quad t \in S^1.$$
Due to the invariance under this group action Rabinowitz action
functional $\mathcal{A}^\mu$ is not Morse but only Morse-Bott. The
critical manifold consists of one single $S^1 \times \mathbb{Z}$
orbit, namely
$$\mathrm{crit}(\mathcal{A}^\mu)=(S^1 \times \mathbb{Z})_*(1,0).$$
Note that since the metric is invariant under the action as well,
the vortex equations itself are $S^1 \times \mathbb{Z}$-invariant.
The action value for the critical point corresponding to $(r,k) \in
S^1\times \mathbb{Z}$ computes to be
$$\mathcal{A}^\mu\big((r,k)_*(1,0)\big)=\pi k$$
namely the area of the unit disk times the winding number around it.
The assumption that the energy is finite guarantees that vortices
exponentially converge to critical points of $\mathcal{A}^\mu$ at
both asymptotic ends. In particular, if $(u,\eta)$ is a vortex there
exist $(r^\pm,k^\pm) \in S^1 \times \mathbb{Z}$ such that
$$\lim_{s \to \pm \infty}(u,\eta)(s)=(r^\pm,k^\pm)_*(1,0).$$
Since the action is nonincreasing along gradient flow lines, we
observe that
\begin{equation}\label{aswin}
k^- \geq k^+.
\end{equation}
Alternatively, this fact can also be
deduced via positivity of intersections by interpreting $-k^\pm$ as
asymptotic winding numbers of the vortex. We further note, that the
inequality is strict, unless the vortex is constant.

\subsection{Transversality}

In this section we show that the standard complex structure on
$\mathbb{C}$ given by multiplication with $i$ is regular, i.e. the
linearization of the vortex equation at each vortex is surjective.
Since Rabinowitz action functional is only Morse-Bott we have to
consider the linearization in suitable weighted Sobolev spaces in
order that it becomes a Fredholm operator. Because critical points
of $\mathcal{A}^\mu$ consist of a single $S^1 \times
\mathbb{Z}$-orbit, the spectrum of the Hessian is independent of the
critical point. We choose $\delta>0$ smaller then the spectral gap
at zero, i.e. smaller then the minimum of the absolute value of all
nonzero eigenvalues of the Hessian. We further choose a smooth
function $\beta \in C^\infty(\mathbb{R},[-1,1])$, for which there
exist $T>0$ with the property that
$$\beta(s)=\left\{\begin{array}{cc}
-1 & s<-T\\
1 & s>T
\end{array}\right.$$
We define
$$\gamma_\delta \in C^\infty(\mathbb{R},\mathbb{R}), \qquad
\gamma_\delta(s)=e^{\beta(s) \delta},\,\, s\in \mathbb{R}.$$ We
abbreviate
$$W^{1,2}_\delta=\big\{f \in
W^{1,2}_{loc}: f\gamma_\delta \in W^{1,2}\big\}$$ the space of all
$W^{1,2}$-functions which at both asymptotics exponentially decay
with weight at least $\delta$. Note that this definition is
independent of the choice of the function $\beta$. Using these
spaces the linearization along a vortex gives rise to a Fredholm
operator
$$D=D_{(u,\eta)} \colon W_{-\delta}^{1,2}(\mathbb{R}\times S^1,\mathbb{C}) \times
W_{-\delta}^{1,2}(\mathbb{R},\mathbb{R}) \to
L^2_{-\delta}(\mathbb{R}\times S^1,\mathbb{C}) \times
L^2_{-\delta}(\mathbb{R},\mathbb{R})$$ given for
$(\hat{u},\hat{\eta}) \in W_{-\delta}^{1,2}(\mathbb{R}\times
S^1,\mathbb{C}) \times W_{-\delta}^{1,2}(\mathbb{R},\mathbb{R})$ by
$$D(\hat{u},\hat{\eta})=\left\{\begin{array}{c}
\partial_s \hat{u}+i\partial_t \hat{u}-2 \pi\eta \hat{u}-2 \pi u\hat{\eta}\\
\partial_s \hat{\eta}-2\pi \int u \hat{u}.
\end{array}\right.
$$

\begin{prop}\label{transverse}
Along any vortex the Fredholm operator $D$ is surjective.
\end{prop}
\textbf{Proof: }Showing surjectivity of $D$ is equivalent to showing
injectivity for the adjoint operator $D^*$. The adjoint operator
$$D^* \colon W_\delta^{1,2}(\mathbb{R}\times S^1,\mathbb{C}) \times
W_\delta^{1,2}(\mathbb{R},\mathbb{R}) \to
L^2_\delta(\mathbb{R}\times S^1,\mathbb{C}) \times
L^2_\delta(\mathbb{R},\mathbb{R})$$  is given by
$$D^*(\hat{u},\hat{\eta})=\left\{\begin{array}{c}
-\partial_s \hat{u}+i\partial_t \hat{u}-2 \pi \eta \hat{u}-2 \pi u\hat{\eta}\\
-\partial_s \hat{\eta}-2\pi \int u \hat{u}.
\end{array}\right.
$$
Assume that
$$(\hat{u},\hat{\eta}) \in \mathrm{ker}D^*.$$
We first show that $\hat{\eta}$ vanishes. For this purpose we
compute using (\ref{vortex})
\begin{eqnarray*}
\partial_s^2 \hat{\eta}&=&-2\pi \int (\partial_s u)\hat{u}
-2\pi \int u (\partial_s \hat{u})\\
&=&2 \pi\int \langle i\partial_t u, \hat{u}\rangle -4\pi^2 \int \eta
\langle u,\hat{u} \rangle\\
& &-2\pi \int \langle u,i\partial_t \hat{u}\rangle +4 \pi^2 \int
\eta \langle u, \hat{u} \rangle +4\pi^2 \int \hat{\eta} \langle u,
u\rangle\\
&=&4\pi^2 \hat{\eta} \int_0^1|u|^2dt
\end{eqnarray*}
Taking the product of this expression with $\hat{\eta}$ and
integrating over $\mathbb{R}$ we obtain via integration by parts
$$4\pi^2 \int_{-\infty}^\infty |\hat{\eta}(s)|^2 \bigg(\int_0^1
\hat{u}(s,t)|dt\bigg)ds=-\int_{-\infty}^\infty |\partial_s
\hat{\eta}|^2 ds.$$ The lefthandside is nonnegative and the
righthandside is nonpositive therefore both sides have to vanish and
we conclude that $\partial_s \hat{\eta}$ vanishes identically. Since
$\hat{\eta} \in W_\delta^{1,2}(\mathbb{R},\mathbb{R})$ we conclude
$$\hat{\eta}=0.$$
Using again that $(\hat{u},\hat{\eta})$ is in the kernel of $D^*$ we
conclude that $\hat{u}$ is a solution of the PDE
$$-\partial_s \hat{u}+i \partial_t \hat{u}-2\pi \eta \hat{u}=0.$$
We expand $\hat{u}$ into a time dependent Fourierseries
$$\hat{u}(s,t)=\sum_{k \in \mathbb{Z}}a_k(s) e^{2\pi i kt}.$$
The Fouriercoefficients are solutions of the ODE
$$\partial_s a_k+2\pi(k+\eta)a_k=0.$$
Asymptically $\eta$ converges to minus the asympotic winding numbers
of the vortex $(u,\eta)$
$$\lim_{s \to \pm \infty}\eta(s)=k^\pm.$$
which satisfy $k^- \geq k^+$ by (\ref{aswin}). We claim that
\begin{equation}\label{coef}
a_k=0, \quad k \in \mathbb{Z}.
\end{equation}
Otherwise, since $a_k$ decays exponentially at the positive end we
would obtain
$$k>-k^+.$$
Since $a_k$ also decays exponentially at the negative end the
inequality
$$k<-k^-$$
has to hold true as well. Together we conclude
$$k^+<k^-$$
contradicting (\ref{aswin}). This proves (\ref{coef}) and therefore
$$\hat{u}=0.$$
We have shown that $(\hat{u},\hat{\eta})=(0,0)$ and the proof of the
Proposition is complete. \hfill $\square$

\subsection{The vortex number}

If the asymptotic winding numbers $k^\pm$ of a vortex $(u,\eta)$
satisfy
$$k^-=k^++1,$$
then the Fredholm index of the Fredholm operator $D=D_{(u,\eta)}$ is
$$\mathrm{ind}(D)=3.$$
We denote by $\widetilde{\mathfrak{V}}(k^-,k^+)$ the moduli space of
all vortices with asymptotic winding numbers $k^-$ and $k^+$. By
Proposition~\ref{transverse} it is a three dimensional manifold. The
group $\mathbb{R}$ acts on vortices by timeshift and the group $S^1$
acts on vortices by rotation of the domain of the loop as well as by
rotation on the target $\mathbb{C}$. Since all this actions commute
we get an $\mathbb{R}\times S^1 \times S^1$ action on
$\widetilde{\mathfrak{V}}(k^-,k^+)$. Since the two asympotics are
different, the action is free. Therefore the quotient
$$\mathfrak{V}(k^-,k^+)=\widetilde{\mathfrak{V}}(k^-,k^+)/
(\mathbb{R}\times S^1 \times S^1)$$ is a zero dimensional manifold.
Since it is compact, see \cite{cieliebak-frauenfelder, frauenfelder}
it is a finite set and we define the vortex number as
$$\mathfrak{v}=
\#\mathfrak{V}(k^-,k^+)\,\,\mathrm{mod}\,\,2 \in \mathbb{Z}_2.$$
Note that since the vortex equation is invariant under the
$\mathbb{Z}$-action the vortex number does not depend on the
asymptotic winding numbers. The vortex number has the following
interpretation. If $(r^-,k^-)$ and $(r^+,k^+)$ are two elements in
$S^1 \times \mathbb{Z}$ satisfying $k^-=k^++1$ then $\mathfrak{v}$
is the modulo 2 number of vortices $(u,\eta)$ subject to the
asymptotic conditions
$$\lim_{s \to \pm \infty}(u,\eta)(s)=
(r^\pm,k^\pm)_*(1,0) \in \mathrm{crit}(\mathcal{A}^\mu).$$

\begin{thm}\label{vonu}
The vortex number equals one.
\end{thm}
\textbf{Proof: } Since the unit circle in the complex plane is
Hamiltonian displaceable, Rabinowitz Floer homology vanishes by
\cite{cieliebak-frauenfelder}. If the vortex number were zero,
Rabinowitz Floer homology would be equal to the homology of the
critical manifold of $\mathcal{A}^\mu$ which is a countable disjoint
union of circles. Therefore the vortex number has to be equal to
one. For an alternative argument based on finite dimensional
approximation we refer to \cite{frauenfelder}.\hfill $\square$

\section{Square roots in simple groups}

In this section we prove that in the Hamiltonian symplectomorphism
group of a closed connected symplectic manifold each element can be
written as a finite product of elements which admit a square root.
This result is a straightforward consequence of a deep result due to
Banyaga.

\begin{thm}\label{ban}
Assume $(M,\omega)$ is a closed connected symplectic manifold. Then
for each $\phi \in \mathrm{Ham}(M,\omega)$ there exists $n \in
\mathbb{N}$ and $\psi_i \in \mathrm{Ham}(M,\omega)$ for $1 \leq i
\leq n$ such that
$$\phi=\psi_1^2 \cdots \psi_n^2.$$
\end{thm}
Before embarking on the proof of the Theorem we first consider an
arbitray group $G$. We denote by $G^2$ the subgroup of $G$ generated
by all squares in $G$, i.e.
$$G^2=\big\{g_1^2 \cdots
g_n^2: g_1,\cdots ,g_n \in G,\,\,n \in \mathbb{N}\big\}.$$
\begin{lemma}\label{g2}
$G^2$ is a normal subgroup in $G$.
\end{lemma}
\textbf{Proof: }Let $g \in G$ and $g_1^2 \cdots g_n^2 \in G^2$. Then
$$g g_1^2\cdots g_n^2 g^{-1}=(g g_1 g^{-1})^2 \cdots (g g_n
g^{-1})^2 \in G^2.$$ This finishes the proof of the Lemma. \hfill
$\square$
\begin{cor}\label{g3}
If $G$ is simple and not two-torsion, then $G^2=G$.
\end{cor}
\textbf{Proof: } Since $G$ is not two-torsion $G^2$ is nontrivial.
Therefore by Lemma~\ref{g2} $G^2$ is a nontrivial normal subgroup of
$G$. Since $G$ is simple $G^2=G$. \hfill $\square$
\\ \\
\textbf{Proof of Theorem~\ref{ban}: } We can assume without loss of
generality that $M$ has positive dimension and therefore
$\mathrm{Ham}(M,\omega)$ is not two-torsion. By Banyaga's theorem,
see \cite{banyaga}, the Hamiltonian symplectomorphism group is
simple. Hence the Theorem follows from Corollary~\ref{g3}. \hfill
$\square$

\end{document}